\documentclass[11pt,a4paper]{article}
\usepackage{amsmath}
\usepackage{appendix}
\usepackage{multirow}
\usepackage{url}
\usepackage{graphicx}
\usepackage[all]{xy}
\usepackage[utf8]{inputenc}
\usepackage{amsfonts}

\usepackage{amssymb}
\date{}
\usepackage{amsthm}
\usepackage{amssymb}
\theoremstyle{plain}
\newtheorem{proposition}{Proposition}[section]
\newtheorem{theorem}{Theorem}[section]
\newtheorem{definition}{Definition}[section]

\newtheorem{remark}{Remark}[section]
\newtheorem{corollary}{Corollary}[section]
\newtheorem{example}{Example}[section]
\numberwithin{equation}{section}
\usepackage{fullpage}
\begin{document}
\begin{center} {\bf Some consequences of Glaisher's map and a generalization of Sylvester's theorem}\end{center}
\begin{center}
 Darlison Nyirenda$^{1}$\,\,\,
and \,
Molatelo Rapudi$^{2}$
 \vspace{0.5cm} \\
$^{1}$  The John Knopfmacher Centre for Applicable Analysis and Number Theory, 	University of the Witwatersrand, P.O. Wits 2050, Johannesburg, South Africa.\\
$^{2}$ School of Mathematics, University of the Witwatersrand, P.O. Wits 2050, Johannesburg, South Africa.\\
e-mails: darlison.nyirenda@wits.ac.za, 942365@students.wits.ac.za\\

\end{center}
\begin{abstract}
Send to AADM, Mathematica Slovoca (Patricks article), etc...
For positive integers $k, l \geq 2$, the set of $k$-regular partitions in which parts appear at most $l$ times has attracted a lot of interest in that a composition of Glaisher's mapping can be used to prove the associated partition identities in certain cases. We consider special cases and derive some arithmetic properties. Of particular focus is the set of partitions in which parts are odd and distinct ($k =2$, $l = 2$). Sylvester proved that, for fixed weight, this set of partitions is equinumerous with the set of self-conjugate partitions.  We introduce a new class of partitions that generalizes self-conjugate partitions and as a result, we extend Sylvester's theorem. Furthermore, using this class of partitions, we give new combinatorial interpretation of some Rogers-Ramanujan identities which were previously considered by A. K. Agarwal.
\end{abstract}
\section{Introduction}
A partition of $n\in \mathbb{Z}_{>0}$ is sequence $\lambda = \lambda = (\lambda_1, \lambda_2, \ldots, \lambda_{\ell})$ of positive integers such that $\sum\limits_{i = 1}^{\ell} \lambda_i = n$ and $\lambda_1 \geq \lambda_{2} \geq \lambda_{3} \geq \cdots \geq \lambda_{\ell}$. The summand $\lambda_i$ is called a part of partition and the sum of summands is the weight. The weight of $\lambda$ is denoted by $\vert \lambda \vert$.  There is so much study of restricted partition functions, for example, see \cite{amn,andrews, munagi, nyirenda1, DN2}. In \cite{aland}, G. Andrews recalls a theorem of Schur, and also reports a hint from K. Alladi. With this hint, the following partition identity is established. 
\begin{theorem}\label{main}
The number of partitions of $n$ into distinct parts not divisible by 3 is equal to the number of partitions of $n$ into odd parts occuring not more than twice.  
\end{theorem}
Actually, it turns out that there is a more general theorem beyond Theorem \ref{main} that can be viewed as a consequence of Glaisher's theorem \cite{glaisher}.  More precisely, for integers $p, k \geq 2$, let $b(n,p,k)$ denote the number of partitions of $n$ into parts not divisible by $p$ and occurring at most $k - 1$ times and let $c(n,k,p)$ denote the number of partitions of $n$ into parts not divisible by $k$ and occurring at most $p-1$ times. It follows that 
 \begin{equation}\label{theomain}
 b(n,p,k) = c(n,k,p).
\end{equation}
This observation was stated in \cite{keith} and attempts to provide a bijective proof that is purely a composition of Glaisher mappings have been made. For the case where $\gcd(k,p)= 1$, one can compose Glaisher maps to establish the bijection. A bijection  for the general case has recently been given by the same author in \cite{keith2}. We shall denote by $\phi_{k,p}$ a bijection from $b(n,p,k)$-partitions onto $c(n,k,p)$-partitions.  In Section 2, we  study a certain class of partitions involving $b(n,pt,p)$-subpartitions and the special case $b(n,3,2)$. In Section 3, we look at $b(n,2,2)$ and introduce a new class of partitions, which we call symmetric partitions.  We generalize Sylvester's theorem for $b(n,2,2)$ via symmetric partitions and consequently, present new interpretations of some Rodgers-Ramanujan identities.\\\\
\noindent We shall denote by $B(n,p,k)$ the set of partitions enumerated by $b(n,p,k)$. Similarly, the same applies to $C(n,k,p)$. For the number of partitions of $n$ into parts appearing at most $k - 1$ times, we use the notation $r(n,k)$, and the underlying set of partitions is denoted by $R(n,k)$,
We also recall the  following notation; for any $a, q\in\mathbb{C}$, we have
$ (a;q)_{n} = \prod\limits_{i = 0}^{n - 1}(1 - aq^{i})$ for $n\geq 1$ and  $(a;q)_{0} = 1$, and also
for $\vert q\vert  < 1$, $(a;q)_{\infty} = \prod\limits_{i = 0}^{\infty}(1 - aq^{i})$. \\
The following $q$-identities will be useful:
\begin{equation}\label{q1}
\sum_{n=0}^{\infty} \frac{t^n q^{\frac{n(n-1)}{2}}}{(q ,q)_n} = \prod_{n=0}^{\infty}(1 + tq^n)\,\,\text{where}\,\,\vert t \vert,\vert q\vert < 1,
\end{equation}
\begin{equation}\label{J}
\sum\limits_{n = -\infty}^{\infty}w^{n}q^{\frac{n(n+1)}{2}} = \prod\limits_{n = 1}^{\infty}(1 - q^{n})(1 + wq^{n})(1 + w^{-1}q^{n-1}),
\end{equation}
\begin{equation}\label{pent}
1 + \sum\limits_{n = 1}^{\infty}(-1)^{n}q^{\frac{n(3n \pm 1)}{2}} = \prod\limits_{n = 1}^{\infty}(1 - q^{n}),
\end{equation}
\begin{equation}\label{cauchy}
\prod\limits_{n=1}^{\infty}\frac{1-q^n}{1+q^n}\, = \, \sum\limits_{n = - \infty}^{\infty}(-1)^{n}q^{n^2}.
\end{equation}
For these $q$-identities, one may refer to \cite{andrews}.
\section{Certain partitions involving $b(n,pt,p)$-subpartitions}
\noindent Consider partitions in $B(n,pt,p)$. We investigate the case where we allow parts divisible by $p$ appear with certain multiplicities. \\
\noindent Let $t \geq 1$ and $p\geq 2$ be integers and denote by $F(n,p,t)$ the set of partitions of $n$ in which parts that are congruent to $0 \pmod{p}$ occur with multiplicities $jt$, $j=1,2,3,\ldots,p - 1$ and parts not divisible by $p$ occur at most $pt-1$ times, and let $f(n,p,t) = \vert F(n,p,t)\vert $. 
\begin{proposition}\label{000}
For all $n\geq 0$, $f(n,p,t) = r(n,p).$
\end{proposition} 
\begin{proof}
\begin{align*}
\sum\limits_{n=0}^{\infty}f(n,p,t)q^n & = \quad \prod_{n=1}^{\infty}(1 + q^{tpn} + q^{2tpn}+ q^{3tpn}+ \cdots + q^{(p-1)tpn})\\
& \quad \quad \times \prod\limits_{r = 1}^{p-1}\prod\limits_{n = 1}^{\infty}(1 + q^{pn - r} + q^{2(pn - r)} + \cdots + q^{(pt-1)(pn - r)})\end{align*}
\begin{align*}
& = \quad \prod_{n=1}^{\infty}\frac{1-q^{p^2tn}}{1-q^{ptn}} \prod\limits_{r = 1}^{p-1}\prod\limits_{n = 1}^{\infty}\frac{1-q^{pt(pn-r)}}{1-q^{pn-r}}\\
& = \quad \prod_{n=1}^{\infty}\frac{(1-q^{p^2tn})}{(1-q^{ptn})}\cdot \frac{(1-q^{pn})}{(1-q^n)}\cdot \frac{(1-q^{ptn})}{(1-q^{p^2tn})}\\
& = \quad \prod_{n=1}^{\infty}\frac{1-q^{pn}}{1-q^n} \\
& = \quad \prod_{n=1}^{\infty} (1 + q^{n} + q^{2n} + q^{3n}+ \cdots + q^{(p-1)n}).
\end{align*}
\end{proof}
We exhibit a bijection from between the sets $F(n,p,t)$  and $R(n,p)$.\\  

\noindent Let $\lambda \in F(n,p,t)$. Decompose $\lambda$ as $\lambda = (\lambda_{p},\lambda_{r})$ where $\lambda_{r}$ is the subpartition of $\lambda$ consisting of all parts not divisible by $p$ and $\lambda_p$ is that subpartition of $\lambda$ with parts divisible by $p$.\\

\noindent Note that $$\lambda_p = (\lambda_1^{j_1t}, \lambda_2^{j_2t}, \lambda_3^{j_3t},\ldots, \lambda_l^{j_lt})$$ where $\lambda_i \equiv 0\pmod{p}$, $j_i = 1,2,\ldots, p - 1$ for all $i$ and $\lambda_1 > \lambda_2 > \lambda_3 > \ldots >\lambda_l$. Then the map
$$ \lambda \mapsto \mu = \phi_{p, pt}(\lambda_{r}) \cup \left((t\lambda_1)^{j_1}, (t\lambda_2)^{j_2}, (t\lambda_3)^{j_3}, \ldots, (t\lambda_l)^{j_l}\right)$$
 is a bijection.\\
\noindent To invert the mapping above, let $\mu \in R(n,p)$. Decompose $\mu$ as $\mu = (\mu_1,\mu_2)$ where $\mu_1$ is the subpartition of $\mu$ consisting of parts congruent to $0 \pmod{pt}$ and $\mu_2$ is the subpartition of $\mu$ with parts not divisible by $pt$ .\\
\noindent Observe that, $\mu_1$ can be written as
$$\mu_1=(\alpha_1^{m_1},\alpha_2^{m_2},\alpha_3^{m_3}, \ldots, \alpha_s^{m_s})$$ for some $\alpha_{i} \equiv 0 \pmod{pt}$  where  $\alpha_1 > \alpha_2 > \alpha_3 > \cdots > \alpha_s \geq 1 $ with $0\leq m_{i} \leq p - 1$. Then the map
$$\mu \mapsto \beta = \phi_{p, pt}^{-1}(\mu_2) \cup \left( \bigcup\limits_{i=1}^{s}\left(\frac{\alpha_i}{t}\right)^{m_it} \right)$$
is the inverse.

\noindent By the proposition above, we have
$$r(n,p) = f(n,p,1) = f(n,p,2) = f(n,p,3) = \cdots $$ so that
\begin{equation}\label{rr1}
 \sum_{t = 1}^{\ell}f(n,p,t)  = \ell r(n,p).
\end{equation}
Thus  from \eqref{rr1}, it is clear that for any $\ell \geq 2$, 
\begin{equation}
\sum_{t = 1}^{\ell}f(n,p,t) \equiv 0 \pmod{\ell}.
\end{equation}
\noindent In reference to Proposition \ref{000}, it is clear that $f(n,2,t)$ ($p = 2$) is equal to the number of partitions of $n$ into distinct parts. Let $f_{e}(n,t)$ (resp. $f_{o}(n,t)$) denote the number of $f(n,2,t)$-partitions of $n$ with an even (resp. odd) number of \textit{different} even parts. We have the following formula for $f_{e}(n,2,t)$ in the corollary below.
\begin{corollary}\label{000isaa}
For $n \geq 0$, we have
\begin{equation}\label{fe2,1}
f_{o}(n,2,t) =  \sum\limits_{j = 1}^{\lfloor \sqrt{\frac{n}{2t}}\rfloor}(-1)^{j+1}d(n-2tj^2),
\end{equation}
where $d(n)$ is the number of partitions of $n$ into distinct parts and $d(0):= 1$.
\end{corollary}

\begin{proof}
\begin{align*}
& \sum\limits_{n=0}^{\infty}(f_{e}(n,2,t) - f_{o}(n,2,t))q^n\\
& = \prod_{n=1}^{\infty}(1 - q^{2tn})\prod\limits_{n = 1}^{\infty}(1 + q^{2n - 1} + q^{2(2n - 1)} + \cdots + q^{(2t-1)(2n - 1)})\\
& = \prod_{n=1}^{\infty}\frac{1-q^{2tn}}{1+q^{2tn}}\prod_{n=1}^{\infty}(1+q^{2tn})(1 + q^{2n - 1} + q^{2(2n - 1)} + \cdots + q^{(2t-1)(2n - 1)})\\
& = \prod_{n=1}^{\infty}\frac{1-q^{2tn}}{1+q^{2tn}}\sum_{n=0}^{\infty}d(n)q^n\\
& = \sum_{n=-\infty}^{\infty}(-1)^nq^{2tn^2}\sum_{n=0}^{\infty}d(n)q^n \,\,\, \text{(by \eqref{cauchy})}.
\end{align*}
\noindent Thus we have
\begin{align*}
& \sum\limits_{n=0}^{\infty}(f_e(n,2,t) + f_o(n,2,t))q^n  \\
& = \left(1+2\sum_{n=1}^{\infty}(-1)^nq^{2tn^2}\right) \sum_{n=0}^{\infty}d(n)q^n + 2\sum\limits_{n=0}^{\infty}f_o(n,2,t))q^n \\
                                                        & = \sum_{n=0}^{\infty}d(n)q^n + 2\sum_{n=1}^{\infty}(-1)^nq^{2tn^2}\cdot \sum_{n=0}^{\infty}d(n)q^n + 2\sum\limits_{n=0}^{\infty}f_o(n,2,t))q^n,
                                                        \end{align*} 
i.e.
$$
\sum_{n=0}^{\infty}d(n)q^n = \sum_{n=0}^{\infty}d(n)q^n + 2\sum_{n=0}^{\infty}\sum_{k=1}^{n}a_k d(n-k)q^n + 2\sum\limits_{n=0}^{\infty}f_{o}(n,2,t))q^n
$$
where
$$
a_k = \begin{cases}
(-1)^j, & k = 2tj^2, j \geq 1;\\
0 , & \text{otherwise}.
\end{cases}
$$
Hence, $$ \sum\limits_{n=0}^{\infty}f_o(n,2,t))q^n  = \sum_{n=0}^{\infty} \left( \sum_{j=1}^{\lfloor \sqrt{\frac{n}{2t}}\rfloor}(-1)^jd(n-2tj^2)\right) q^n. $$
\noindent By comparing the coefficients of $q^n$ on both sides, we obtain the result.
\end{proof}
Let $p = 3$ and $k = 2$ in Theorem \ref{theomain}. We then get a partition function $b(n,3,2)$, the number of 3-regular partitions of $n$ into distinct parts. Further, let $c(n,i)$ to be the number of partitions of $n$ in which parts are not congruent to $0, \pm i \pmod{9}$. Theorem \ref{rela} shows how $c(n,i)$ and $b(n,3,2)$ are related in special cases. 
\begin{theorem}\label{rela}
We have:
\begin{align*}
\sum\limits_{n = 0}^{\infty}b(n,3,2)q^{n} & =  \frac{(q^{27};q^{27})_{\infty}(-q^{15};q^{27})_{\infty}(-q^{12};q^{27})_{\infty}}{ (q^3; q^{3})_{\infty} } \\
& \,\,\,\,\,\,\,\,\,\,\,\,\,\,\,\, + q\frac{(1 + q^{6})(q^{27};q^{27})_{\infty} (-q^{21};q^{27})_{\infty}(-q^{33};q^{27})_{\infty}}{(q^3; q^{3})_{\infty}}\\
&\,\,\,\,\,\,\,\,\,\,\,\,\,\,\,\,\,\,\,\,\,\,\,\,\,\,\,\,\,\,\, + q^{2}\frac{(q^{27};q^{27})_{\infty}(-q^{3};q^{27})_{\infty} (-q^{24};q^{27})_{\infty}}{(q^3; q^{3})_{\infty}}.
\end{align*}
Consequently, for all $n\geq0$,
$$ b(3n,3,2) \equiv c(n,4) \pmod{2}$$
and
$$b(3n + 2,3,2) \equiv c(n,1) \pmod{2}.$$
\end{theorem}
\begin{proof}
 Clearly, 
$$\sum\limits_{n  = 0}^{\infty}b(n,3,2)q^{n} = \prod\limits_{n = 1}^{\infty}(1 + q^{3n - 1})(1 + q^{3n - 2}).$$ which implies
 $$\sum\limits_{n  = 0}^{\infty}b(n,3,2)q^{n} = \frac{1}{\prod\limits_{n = 1}^{\infty}(1 - q^{3n})}\prod\limits_{n = 1}^{\infty}(1- q^{3n})(1 + q^{3n - 1})(1 + q^{3n - 2})$$
 so that
$$\sum\limits_{n  = 0}^{\infty}b(n,3,2)q^{n} = \frac{1}{(q^3; q^{3})_{\infty}}\sum\limits_{ n = -\infty}^{\infty}q^{\frac{3n(n + 1)}{2} - n}.$$
\noindent This follows by setting $q:= q^{3}$ and $w: = q^{-1}$ in \eqref{J}.
Thus
\begin{align*}
& \sum\limits_{n  = 0}^{\infty}b(n,3,2)q^{n} \\
& =  \frac{1}{(q^3; q^{3})_{\infty}}\sum\limits_{ n = -\infty}^{\infty}q^{\frac{3n^2 + n}{2}} \\
                                                            & =  \frac{1}{(q^3; q^{3})_{\infty}}\left( \sum\limits_{ n = -\infty}^{\infty}q^{\frac{3(9n^2) + 3n}{2}} + \sum\limits_{ n = -\infty}^{\infty}q^{\frac{3(3n +2)^2 + 3n + 2}{2}} + \sum\limits_{ n = -\infty}^{\infty}q^{\frac{3(3n +1)^2 + 3n + 1}{2}}\right)\\
& =  \frac{1}{(q^3; q^{3})_{\infty}}\left( \sum\limits_{ n = -\infty}^{\infty}q^{\frac{27n^2 + 3n}{2}} + \sum\limits_{ n = -\infty}^{\infty}q^{\frac{27n^2 + 39n + 14}{2}} + \sum\limits_{ n = -\infty}^{\infty}q^{\frac{27n^2 + 21n + 4}{2}}\right)\\
& = \frac{1}{(q^3; q^{3})_{\infty}}\left( \sum\limits_{ n = -\infty}^{\infty}q^{\frac{27n^2 + 3n}{2}} + q\sum\limits_{ n = -\infty}^{\infty}q^{\frac{27n^2 + 39n}{2} + 6} + q^{2}\sum\limits_{ n = -\infty}^{\infty}q^{\frac{27n^2 + 21n}{2}}\right)\end{align*}
\begin{align*}
& = \frac{1}{(q^3; q^{3})_{\infty}}\left( (q^{27};q^{27})_{\infty}(-q^{15};q^{27})_{\infty}(-q^{12};q^{27})_{\infty}\right.  \\ 
 & \left. \quad \quad \quad \,\,\,+ q^{7}(1 + q^{-6})(q^{27};q^{27})_{\infty} (-q^{21};q^{27})_{\infty}(-q^{33};q^{27})_{\infty} \right.\\
&\left.  \quad \quad \quad \,\,\,+ q^{2}(q^{27};q^{27})_{\infty}(-q^{3};q^{27})_{\infty} (-q^{24};q^{27})_{\infty} \right)\\
& = \frac{1}{(q^3; q^{3})_{\infty}}\left( (q^{27};q^{27})_{\infty}(-q^{15};q^{27})_{\infty}(-q^{12};q^{27})_{\infty}\right.  \\ 
 & \left. \quad \quad \quad \,\,\,+ q(1 + q^{6})(q^{27};q^{27})_{\infty} (-q^{21};q^{27})_{\infty}(-q^{33};q^{27})_{\infty} \right.\\
&\left.  \quad \quad \quad \,\,\,+ q^{2}(q^{27};q^{27})_{\infty}(-q^{3};q^{27})_{\infty} (-q^{24};q^{27})_{\infty} \right)
\end{align*}
which follows by invoking \eqref{J} and a bit of algebra. Thus, we have the dissection:
\begin{align*}
\sum\limits_{n = 0}^{\infty}b(n,3,2)q^{n} & =  \frac{(q^{27};q^{27})_{\infty}(-q^{15};q^{27})_{\infty}(-q^{12};q^{27})_{\infty}}{ (q^3; q^{3})_{\infty} }\\
& \,\,\,\,\,\,\,\,\,\,\,\,\,\,\,\, + q\frac{(1 + q^{6})(q^{27};q^{27})_{\infty} (-q^{21};q^{27})_{\infty}(-q^{33};q^{27})_{\infty}}{(q^3; q^{3})_{\infty}}\\
&\,\,\,\,\,\,\,\,\,\,\,\,\,\,\,\,\,\,\,\,\,\,\,\,\,\,\,\,\,\,\, + q^{2}\frac{(q^{27};q^{27})_{\infty}(-q^{3};q^{27})_{\infty} (-q^{24};q^{27})_{\infty}}{(q^3; q^{3})_{\infty}}
\end{align*}
From the 3-dissection above, it follows that
\begin{equation}\label{diss1}
\sum\limits_{n = 0}^{\infty}b(3n,3,2)q^{3n}  =  \frac{(q^{27};q^{27})_{\infty}(-q^{15};q^{27})_{\infty}(-q^{12};q^{27})_{\infty}}{ (q^3; q^{3})_{\infty} }
\end{equation}
and
\begin{equation}\label{diss3}
\sum\limits_{n = 0}^{\infty}b(3n + 2,3,2)q^{3n + 2}  = q^{2}\frac{(q^{27};q^{27})_{\infty}(-q^{3};q^{27})_{\infty} (-q^{24};q^{27})_{\infty}}{(q^3; q^{3})_{\infty}}.  
\end{equation}
 Using \eqref{diss1},
\begin{align*}
\sum\limits_{n = 0}^{\infty}b(3n,3,2)q^{n}  & =  \frac{(q^{9};q^{9})_{\infty}(-q^{5};q^{9})_{\infty}(-q^{4};q^{9})_{\infty}}{ (q; q)_{\infty} }  \\
                                                               & \equiv \frac{(q^{9};q^{9})_{\infty}(q^{5};q^{9})_{\infty}(q^{4};q^{9})_{\infty}}{ (q; q)_{\infty} }  \pmod{2}
\end{align*}
and the first parity identity in the theorem follows. Similarly, the other one arises from \eqref{diss3}.
\end{proof}
If we let $k = p = 2$ in Theorem \ref{theomain}, we get $b(n,2,2)$ which enumerates partitions into distinct odd parts. Actually, J. J. Sylvester proved the following result, which we recall here:
\begin{theorem}[Sylvester, \cite{av}]\label{self}
For all $n\geq 0$, $b(n,2,2)$ is equal to the number of self-conjugate partitions.
\end{theorem}
There is a bijective proof which goes as follows (see \cite{av}).\\
Let $\lambda = (\lambda_1, \lambda_2, \ldots )$ be a self-conjugate partition of $n$. Then the mapping
$$ \lambda \mapsto (2\lambda_{1} - 1, 2\lambda_{2} - 3,\ldots, 2\lambda_{s} - 2s + 1),$$
where $s$ is the order of $\lambda$, is a bijection from the set of self-conjugate partitions of $n$ onto the set of partitions of $n$ into distinct odd parts. In the next section, we generalize Sylvester's theorem, Theorem \ref{self}.  
\section{Extending Sylvester's theorem}\label{sec2}
In this section we introduce a new class of partitions called symetric partitions and how they relate to certain types of partitions in the literature. However, we shall require a statistic called  the order of a partition (see \cite{av}).
\begin{definition}\label{order}
For a partition $\lambda = (\lambda_1, \lambda_2, \lambda_3, \ldots)$ where $\lambda_1 \geq \lambda_2 \geq \lambda_3 \geq \cdots$, we define the \textit{order} of $\lambda$ to be the integer $ \max\{i:\lambda_{i} \geq i\}$.
\end{definition}
 \noindent Our new class of partitions which we shall call ($\mu, \gamma$)-symmetric partitions is defined as follows.
\begin{definition}\label{mol1}
Let $\mu \geq 2$ and $\gamma \geq 0$ be integers. A partition $\lambda = (\lambda_1, \lambda_2, \lambda_3, \ldots)$ is said to be $(\mu, \gamma)$-symmetric if the tail $(\lambda_{s + 1}, \lambda_{s + 2}, \ldots )$ is equal to 
$$
\begin{cases} (s^{(\mu - 1)(\lambda_{s} - s) + \gamma},(s - 1)^{(\mu - 1)(\lambda_{s - 1} - \lambda_{s} + 1) - 1},\ldots ,1^{(\mu - 1)(\lambda_1 - \lambda_2 + 1) - 1}), &\text{if $s > 1$};\\
s^{(\mu - 1)(\lambda_{s} - s) + \gamma},& \text{if $s = 1$}.
\end{cases}
$$
where $s$ is the order of $\lambda$.
\end{definition}
\noindent We denote the number of $(\mu,\gamma)$-symmetric partitions of $n$ by $g(n, \mu,\gamma)$.
\begin{example}
The $(2,1)$-symmetric partitions of 10 are $(5,1^{5})$, $(4,2^{2}, 1^{2})$ and $(3^{2}, 2^{2})$. 
\end{example}
We declare the following result.
\begin{theorem}\label{Sem1}
For $n\geq 0$, $g(n,\mu,\gamma)$ is equal to the number of partitions of $n$ into distinct parts congruent to $1 + \gamma \pmod{\mu}$ and the smallest part is at least $\gamma + 1$.
\end{theorem}
\begin{proof}
\begin{align*}
& \sum\limits_{n=0}^{\infty}g(n,\mu,\gamma)q^n  \\
& = \sum_{s\geq 0,}\sum_{\lambda_1 \geq \lambda_2 \geq \cdots \geq \lambda_s \geq s }q^{\mu \sum\limits_{i = 1}^{s}\lambda_i - \mu\frac{s(s + 1)}{2} + s\gamma + s} \\
& = \sum_{s\geq 0,}\sum_{n_s \geq 0,}\sum_{n_{s-1 \geq 0}}\cdots\sum_{n_1 \geq 0} \left( q^{\mu s^2} \cdot (q^{\mu s})^{n_s} \cdot (q^{\mu (s-1)})^{n_{s-1}} \cdots (q^{2\mu})^{n_2} \cdot (q^\mu)^{n_1}\right.\\
& \left.\quad \quad \quad \times q^{- \frac{\mu s(s+1)}{2} + s\gamma + s}\right)\\
& = \sum_{s=0}^{\infty} \frac{q^{\mu s^2- \frac{\mu s(s+1)}{2} + s\gamma + s}}{(1-q^{\mu s})(1- q^{\mu (s-1)})\ldots(1-q^{2\mu})(1-q^\mu)} \\
& = \sum_{s=0}^{\infty} \frac{q^{\frac{\mu s(s-1)}{2} + s (\gamma + 1)}}{(q^\mu ,q^\mu)_s}, \\
\end{align*}
\noindent i.e.
\begin{equation}\label{mugama}
\sum\limits_{n=0}^{\infty}g(n,\mu,\gamma)q^n \quad  = \quad \sum_{n=0}^{\infty} \frac{q^{\frac{\mu n(n-1)}{2} + n (\gamma + 1)}}{(q^\mu ,q^\mu)_n}.
\end{equation}

\noindent Replacing $t$ by $q^{\gamma +1}$ and $q$ by $q^\mu$ in \eqref{q1}, we get\\

$$\sum_{n=0}^{\infty} \frac{q^{\frac{\mu n(n-1)}{2} + n (\gamma + 1)}}{(q^\mu ,q^\mu)_n} = \prod_{n=0}^{\infty}(1 + q^{\mu n + (\gamma + 1)}).$$
\end{proof}

\noindent \textbf{Note}: The $(2,0)$-symmetric partitions are precisely the self-conjugate partitions. Observe that
a $(2,0)$-symmetric partition $\lambda$ is such that
$$ \lambda = (\lambda_1, \lambda_2, \lambda_3, \ldots, \lambda_{s - 1}, \lambda_s, s^{(\lambda_{s} - s)},(s - 1)^{(\lambda_{s - 1} - \lambda_{s})}, \ldots , 1^{(\lambda_1 - \lambda_2)}),$$
where $s$ is the order of $\lambda$  and $\lambda_{s}$ may be equal to $s$. If we conjugate, we get
\begin{align*}
& \lambda^{\prime}\\
 & = \left(( s+ \lambda_s - s + \sum_{i=1}^{s-1}(\lambda_i -\lambda_{i+1}))^{1},( s+ \lambda_s - s + \sum_{i=2}^{s-1}(\lambda_i -\lambda_{i+1})) ^{2-1},( s+ \lambda_s - s + \right.\\
& \quad \quad \quad \left.\sum_{i=3}^{s-1}(\lambda_i -\lambda_{i+1})) ^{3-2},\ldots,( s+ \lambda_s - s) ^{s-(s-1)}, (s)^{\lambda_s - s}, (s-1)^{\lambda_{s-1} - \lambda_s},\ldots, 1 ^{\lambda_1 -\lambda_2}\right) \\
& = (\lambda_1, \lambda_2, \lambda_3, \lambda_4, \ldots, \lambda_{s - 1}, \lambda_s, s^{(\lambda_{s} - s)}(s - 1)^{(\lambda_{s - 1} - \lambda_{s})},\ldots ,1^{(\lambda_1 - \lambda_2)})
\end{align*}
i.e. $\lambda = \lambda^{\prime}$.\\

\noindent On the other hand, those distinct parts congruent to $1 \pmod{2}$  are actually distinct odd parts. Indeed Theorem \ref{Sem1} is an extension of Sylvester's theorem, Theorem \ref{self}.

\subsection{A bijection for Theorem \ref{Sem1}}
\noindent Let $\lambda = (\lambda_1, \lambda_2, \ldots, \ldots )$ be a $(\mu, \gamma)$-symmetric partition of $n$. Then the mapping
\begin{equation}\label{bijection}
\lambda \mapsto \beta = \left(\mu(\lambda_1 - 1) + 1 + \gamma, \mu(\lambda_2 - 2) + 1 + \gamma, \ldots, \mu(\lambda_s - s) + 1 + \gamma\right)
\end{equation}
is an explicit bijection from the set of $(\mu,\gamma)$-symmetric partitions of $n$ onto the set of partitions of $n$ into distinct parts congruent to $1 + \gamma \pmod{\mu}$ and the smallest part is $\geq \gamma + 1$.
Firstly, observe that the parts of $\beta$ are distinct since
$$\beta_{i} - \beta_{i + 1} = \mu(\lambda_i - i) + 1 + \gamma - (\mu(\lambda_{i + 1} - i - 1) + 1 + \gamma) = \mu(\lambda_{i} - \lambda_{i + 1}) + \mu \geq \mu > 1$$
so that $\beta_{i} > \beta_{i+1}$ for all $i$.\\
Secondly, we claim that the map preserves weight since 
\begin{align*}
\vert \beta \vert & = \sum\limits_{i = 1}^{s}\left( \mu(\lambda_i - i) + \gamma + 1\right)\\
& = \mu\sum\limits_{i = 1}^{s}(\lambda_i - i) + \sum\limits_{i = 1}^{s}(\gamma + 1) \\
& = \mu\sum\limits_{i = 1}^{s}\lambda_i\,\, - \mu\frac{s(s + 1)}{2} + s\gamma + s \\
& = \mu\sum\limits_{i = 1}^{s}\lambda_i\,\, - \,\,(\mu - 1)s^2  + s\gamma + \mu\frac{s(s - 1)}{2} - s(s - 1) \end{align*}
\begin{align*}
& =\sum\limits_{i = 1}^{s}\lambda_i  + (\mu - 1)\left( s\lambda_{s} - s^2  + \sum_{j = 1}^{s}\lambda_{j} - s\lambda_s\right)   + s\gamma + (\mu - 2)\frac{s(s - 1)}{2}\\
& = \sum\limits_{i = 1}^{s}\lambda_i + (\mu - 1)(s\lambda_{s} - s^2) + s\gamma + (\mu - 1)\left(\sum_{j = 1}^{s - 1}j(\lambda_{j} - \lambda_{j + 1})\right) + (\mu - 2)\sum_{j = 1}^{s - 1}j\\
& = \sum\limits_{i = 1}^{s}\lambda_i + s((\mu - 1)(\lambda_{s} - s) + \gamma) +   \sum\limits_{i = 1}^{s - 1}j(  (\mu - 1)(\lambda_{j} - \lambda_{ j + 1} + 1) - 1)  \\
& = \vert \lambda\vert.
\end{align*}
\begin{remark} The explicit bijection given above generalises the explicit bijection given for Theorem \ref{self}.
\end{remark}
The inverse of the map above is described as follows. Let $(\beta_1,\beta_{2}, \ldots, \beta_{\ell})$ be a partition of $n$ into distinct parts congruent to $\gamma + 1$ modulo $\mu$ and each part is at least $\gamma + 1$. For $1\leq i\leq \ell$, let
$$ \lambda_{i} =   i  + \frac{\beta_i - 1 - \gamma}{\mu}.$$
Then the partition
 $(\lambda_1, \lambda_2, \lambda_3, \ldots, \lambda_{\ell},\lambda_{\ell + 1}, \lambda_{\ell + 2}, \ldots )$  where

$(\lambda_{\ell + 1}, \lambda_{\ell + 2}, \ldots ) $ defined as:
$$
\begin{cases} (\ell^{(\mu - 1)(\lambda_{\ell} - \ell) + \gamma},(\ell - 1)^{(\mu - 1)(\lambda_{\ell - 1} - \lambda_{\ell} + 1) - 1},\ldots ,1^{(\mu - 1)(\lambda_1 - \lambda_2 + 1) - 1}), &\text{ if $\ell > 1$};\\
\ell^{(\mu - 1)(\lambda_{\ell} - \ell) + \gamma},&\text{if $\ell = 1$}
\end{cases}
$$
 is a $(\mu,\gamma)$-symmetric partition of $n$.
\begin{example}
Consider the case $n = 10$, $\mu = 2$ and $\gamma = 1$.
\end{example}
\noindent The following displays the one-to-one correspondence for this case.
\begin{table}[h!]
    \label{tab:table1}
    \begin{tabular}{lcr}
       partitions into distinct even parts & $\mapsto$ & $(2,1)$-symmetric partitions\\
      \hline
      $(10)$ & $\mapsto $ &  $(5,1^{5})$\\
      $(8,2)$ & $\mapsto$ & $(4,2^2,1^2)$ \\
      $(6,4)$ & $\mapsto$ & $(3^2, 2^2)$\\
      \hline
    \end{tabular}
    \end{table}

\subsection{New combinatorial interpretation of some Rogers-Ramanujan Identities}
In this section, we interpret the following Rogers-Ramanujan identities via symmetric partitions:
\begin{equation}\label{slat1}
\sum\limits_{n=0}^{\infty} \frac{q^{2n^2}}{(q;q)_{2n}} = \prod\limits_{n=1}^{\infty} \frac{(1-q^{8n-1})(1-q^{8n-7})(1-q^{16n-10})(1-q^{16n-6})(1-q^{8n})}{1-q^{n}},
\end{equation}
\begin{equation}\label{slat2}
\sum\limits_{n=0}^{\infty} \frac{q^{2n(n+1)}}{(q;q)_{2n+1}} = \prod\limits_{n=1}^{\infty} \frac{(1-q^{8n-3})(1-q^{8n-5})(1-q^{16n-14})(1-q^{16n-2})(1-q^{8n})}{1-q^n}.
\end{equation}
\noindent The two identities are part of Slater's compendium (see \cite{slater}) and were given an interpretation in terms of $n$-color partitions by A. K. Agarwal (see \cite{c2}).\\
However, before we proceed, it is worth mentioning that the number of $(\mu,\gamma)$-symmetric partitions $g(n,\mu,\gamma)$ is actually the sum of two functions:
$$g(n,\mu,\gamma) = g_{e}(n,\mu,\gamma)\,\, + \, \,g_{o}(n,\mu,\gamma)$$ where $g_{e}(n,\mu,\gamma)$ (resp. $g_{o}(n,\mu,\gamma)$) denotes the number of $g(n,\mu,\gamma)$-partitions with an even (resp. odd) order. Hence, it is clear from \eqref{mugama} that
 \begin{align}\label{inde1}
 \sum\limits_{n=0}^{\infty}g_{e}(n,\mu,\gamma)q^n & = \sum_{j=0}^{\infty} \frac{q^{\frac{\mu 2j(2j-1)}{2} + 2j (\gamma + 1)}}{(q^\mu ,q^\mu)_{2j}} \nonumber\\
                                                 & = \sum_{j=0}^{\infty} \frac{q^{\mu j(2j-1) + 2j (\gamma + 1)}}{(q^\mu ,q^\mu)_{2j}} 
\end{align}
and
\begin{align}\label{inde2}
\sum\limits_{n=0}^{\infty}g_{o}(n,\mu,\gamma)q^n & = \sum_{j=0}^{\infty} \frac{q^{\frac{\mu (2j + 1)(2j)}{2} + (2j + 1)(\gamma + 1)}}{(q^\mu ,q^\mu)_{2j + 1}} \nonumber\\
                                                &  = \sum_{j=0}^{\infty} \frac{  q^{\mu (2j + 1)j + (2j + 1)(\gamma + 1) } }{(q^\mu ,q^\mu)_{2j + 1}}.
\end{align}
The following theorems give new partition-theoretic interpretation of \eqref{slat1} and \eqref{slat2}, respectively.
\begin{theorem}\label{s2}
Let $\alpha \equiv 0\pmod{2}$ be a positive integer. Then $g_{e}(\alpha n,\alpha,\frac{\alpha}{2} - 1)$ is equal to the number of partitions of $n$ into parts not congruent to $ 0, \pm 1, \pm 6, \pm 7, 8 \pmod{16}$. 
\end{theorem}
\begin{proof}
\noindent Replacing $\mu$ by $\alpha$, $\gamma$ by $\frac{\alpha}{2}-1$  in \eqref{inde1}, we have
\begin{align*}
& \sum\limits_{n=0}^{\infty}g_{e}(n,\alpha,\frac{\alpha}{2}-1)q^n \\
& =  \sum_{j=0}^{\infty} \frac{q^{{\alpha j(2j-1)} + j\alpha}}{(q^\alpha;q^\alpha)_{2j}} \\
& = \sum_{j=0}^{\infty} \frac{q^{2\alpha j^2}}{(q^\alpha;q^\alpha)_{2j}} \\
& = \prod_{n=1}^{\infty} \frac{(1-q^{\alpha(8n-1)})(1-q^{\alpha(8n-7)})(1-q^{\alpha(16n-10)})(1-q^{\alpha(16n-6)})(1-q^{8\alpha n})}{(1-q^{\alpha n})}
\end{align*}
\noindent Setting $q: = q^{\frac{1}{\alpha}}$ yields
$$ \sum\limits_{n=0}^{\infty}g_{e}(\alpha n,\alpha,\frac{\alpha}{2} -1)q^n = \prod_{n=1}^{\infty} \frac{(1-q^{8n-1})(1-q^{8n-7})(1-q^{16n-10})(1-q^{16n-6})(1-q^{8n})}{(1-q^n)}.$$
\end{proof}

\begin{example} 
Consider $\alpha = 2$ and $n = 6$.
\end{example}
\noindent The $(2,0)$-symmetric partitions of 12 with an even order are
$$(6,2,1,1,1,1), (5,3,2,1,1), (4,4,2,2)$$
\noindent and partitions of 6 whose parts are not congruent to $ 0, \pm 1, \pm 6, \pm 7, 8 \pmod{16}$ are
$$(4,2), (3,3), (2,2,2).$$
In either case, the number is the same.
\begin{theorem}\label{s3}
Let $\alpha \equiv 0\pmod{2}$ be a positive integer. Then $g_{o}( \alpha n + \frac{\alpha}{2}, \alpha, \frac{\alpha}{2} -1)$ is equal to the number of partitions of $n$ into parts not congruent to $ 0, \pm 2, \pm 3, \pm 5, 8 \pmod{16}$. 
\end{theorem}
\begin{proof}
The proof is similar to the proof of Theorem \ref{s2}. Set $\mu: = \alpha$, $\gamma:= \frac{\alpha}{2}-1$  in \eqref{inde2} so that
\begin{align*}
\sum\limits_{n = 0}^{\infty}g_{o}(n, \alpha, \frac{\alpha}{2} -1)q^{n} & = \sum_{j=0}^{\infty} \frac{  q^{\alpha (2j + 1)j + (2j + 1)(\frac{\alpha}{2}) } }{(q^\alpha ,q^\alpha)_{2j + 1}} \\
                               & = \sum_{j=0}^{\infty} \frac{  q^{2\alpha j (j + 1) + \frac{\alpha}{2}} }{(q^\alpha ,q^\alpha)_{2j + 1}}.
\end{align*}
Thus
\begin{align*}
& \sum\limits_{n = 0 }^{\infty}g_{o}(n + \frac{\alpha}{2}, \alpha, \frac{\alpha}{2} -1)q^{n} \\
& = \sum_{j=0}^{\infty} \frac{  q^{2\alpha j (j + 1)} }{(q^\alpha ,q^\alpha)_{2j + 1}} \\
                            & = \prod_{n=1}^{\infty} \frac{(1-q^{\alpha (8n-3)})(1-q^{\alpha(8n-5)})(1-q^{\alpha(16n-14)})(1-q^{\alpha(16n-2)})(1-q^{8\alpha n})}{1-q^{\alpha n}}
\end{align*}
which yields
\begin{align*}
 &  \sum\limits_{n = 0}^{\infty}g_{o}(\alpha n + \frac{\alpha}{2}, \alpha, \frac{\alpha}{2} -1)q^{n} \\
&   = \prod_{n=1}^{\infty} \frac{(1-q^{8n-3})(1-q^{8n-5})(1-q^{16n-14})(1-q^{16n-2})(1-q^{8n})}{1-q^n}.
\end{align*}
\end{proof}
\section*{Declarations}
\begin{itemize}
\item Funding: No funding was received for conducting this study. The authors have no financial interests in any material discussed in this article.
\item Conflict of interest: On behalf of all authors, the corresponding author states that there is no conflict of interest.
\item Authors' contributions: The authors contributed equally to this work.
\item Data availability: The authors can confirm that this manuscript has no associated data
\end{itemize}

\end{document}